\documentclass[reqno]{amsart}
\usepackage{graphicx} \usepackage{mathtools}
\usepackage{calrsfs} \usepackage{microtype}
\usepackage{enumerate}
\usepackage{xcolor}
\usepackage{comment}
\usepackage{amsfonts} \usepackage{amsmath} \usepackage{amsrefs}
\usepackage{amssymb} \usepackage{amstext} \usepackage{amsthm}

\usepackage{cancel} \usepackage{caption} \usepackage{enumitem}
\usepackage{hyperref} \usepackage{cleveref} \usepackage{subcaption} \usepackage{color}

\theoremstyle{plain}

\newtheorem{theorem}{Theorem}[section]

\newtheorem{corollary}[theorem]{Corollary}
\newtheorem{lemma}[theorem]{Lemma}
\newtheorem{proposition}[theorem]{Proposition}
\theoremstyle{definition}

\theoremstyle{remark} \newtheorem{remark}[theorem]{Remark}
\newtheorem{example}[theorem]{Example}

  \def\cC{\mathcal C}

\def\cH{\mathcal H}  

\def\cV{\mathcal{V}_{\varepsilon}^3}
 
\newcommand{\tr}{\mathrm{Tr}} \newcommand{\PG}{\mathrm{PG}}
\newcommand{\AG}{\mathrm{AG}}

\newcommand{\FF}{\mathbb{F}}
 
\newcommand{\PGL}[2]{\mathrm{PGL}_{#1}(#2)}

\newcommand{\GF}[1]{\mathbb{F}_{#1}}

\begin{document}

\title{ On a class of quasi-Hermitian surfaces in even characteristic }
\author[A. Aguglia]{Angela Aguglia}
\address{Dipartimento di Meccanica, Matematica e Management, Politecnico di Bari (IT)}
\email{angela.aguglia@poliba.it}
\author[A. Montinaro]{Alessandro Montinaro} 
\address{Dipartimento di Matematica e Fisica "Ennio De Giorgi", Università del Salento (IT)}
\email{alessandro.montinaro@unisalento.it}

\date{}

\begin{abstract}

In \cite{A}, a new quasi-Hermitian variety $\mathcal{H}_\varepsilon^r$ in $\mathrm{PG}(r, q^2)$, with $q = 2^e$ and $e \geq 3$ an odd integer, was constructed. The variety depends on a primitive element $\varepsilon$ of the underlying field $\mathrm{GF}(q^2)$.

In the present paper, we first provide a classification of such varieties up to projective equivalence in finite projective spaces of arbitrary dimension. Then, we focus on the case $r = 3$ and study the structure of the lines contained in $\mathcal{H}_\varepsilon^3$; as a consequence, we determine the full automorphism group of $\mathcal{H}_\varepsilon^3$ .
Finally, as a byproduct, we prove the equivalence of certain minimal codes introduced in \cite{ALGS}.
\end{abstract}


\maketitle

\section{Introduction}

In finite geometry, two-character sets have attracted considerable attention due to their distinctive combinatorial properties and wide-ranging applications, including coding theory, strongly regular graphs, association schemes, optimal multiple coverings, and secret sharing schemes (see, for example, \cite{CK, CI, DD}).

A two-character set $S$ in the Desarguesian projective space 
$\PG(r,q)$ 
is a set of points such that every hyperplane intersects 
$S$
in either 
$n_1$ or 
$n_2$ points, for two fixed integers 
$n_1$
and $n_2$.
The geometric construction of such sets often involves advanced techniques, including commuting polarities, symplectic polarities,  normal line-spreads as well as quasi-polar spaces, (see, for example,  \cite{CMPS}, \cite{DHOP},\cite{DS}).

In this paper, we focus on a particular class of two-character sets arising from quasi-Hermitian varieties in $\PG(r,q^2)$.
A quasi-Hermitian variety in $\PG(r,q^2)$ is a point set that shares the same intersection numbers with hyperplanes as the non-singular Hermitian variety 
 $\cH(r,q^2)$ of $\PG(r,q^2)$.
More precisely,  the intersection numbers of a quasi-Hermitian variety with hyperplanes are given by:
\[| \mathcal{H}(r-1, q^2) |= \frac{(q^r + (-1)^{r-1})(q^{r-1}-(-1)^{r-1})}{
q^2-1} ,\]
and
\[|\Pi_0 \mathcal{H}(r-2, q^2)|= \frac{(q^r + (-1)^{r-1})(q^{r-1}-(-1)^{r-1})}{q^2-1} +
(-1)^{r-1}q^{r-1},\]
where $\Pi_i$ denotes an $i$-dimensional space of $\PG(r, q^2)$ and $\Pi_0\mathcal{H}(r-2, q^2)$ is a cone, the join of
the vertex $\Pi_0$ to a non-singular Hermitian variety $\mathcal{H}(r-2, q^2)$ of a projective subspace $\Pi_{r-2}$
disjoint from $\Pi_0$. 

Since there are only two possible intersection sizes with hyperplanes, a quasi-Hermitian variety is a two-character set.

By definition, every non-singular Hermitian variety in $\PG(r, q^2)$ is a quasi-Hermitian variety referred to as the classical quasi-Hermitian variety. 
The total number of points in a non-singular Hermitian variety 
of $\PG(r,q^2)$ is
\[| \mathcal{H}(r, q^2) |= \frac{(q^{r+1} + (-1)^r)(q^r-(-1)^r)}
{(q^2-1)}. \]

For $r>2$ this coincides with the number of points of any quasi-Hermitian variety in $\PG(r, q^2)$, see~\cite{SV}.
In the Desarguesian projective plane $\PG(2,q^2)$ a set of points having the same size (that is, $q^3+1$) and the same intersection numbers ($1$ and $q+1$) with respect to lines as the Hermitian curve  is defined to be a unital. However, a set in $\PG(2,q^2)$ with  intersection numbers  $1$ or $q+1$   has size either $q^3 + 1$ or $q^2 + q + 1$, corresponding respectively to a unital or a Baer subplane of $\PG(2, q^2)$.
Non-classical quasi-Hermitian varieties do exist, as shown in \cite{A,ACK,DS, LLP,FP}.

In 1976, F. Buekenhout introduced a general construction for unitals in finite translation planes of square order~\cite{B}. This construction led to the discovery of non-classical unitals in $\PG(2, q^2)$ for certain values of $q$. Specifically, when $q \geq 8$ is an odd power of $2$, the Bruck–Bose representation of $\PG(2, q^2)$ in $\PG(4,q)$ allows the construction of a unital from an ovoidal cone in $\PG(4, q)$ with vertex a point $V$ and base a Suzuki–Tits ovoid. Such a unital is known as a Buekenhout–Tits unital, or BT unital.

In~\cite{A}, a new family of non-classical quasi-Hermitian varieties in $\PG(r, q^2)$ is introduced. These varieties, called BT quasi-Hermitian varieties, can be viewed as a natural generalization of BT-unitals to higher dimensions, since for $r=2$ they coincide with the BT-unitals, (for more details, see Section \ref{prelim}).

  All BT unitals are equivalent under the action of $P\Gamma L(3,q^2)$ as proven in \cite{FG},
 resolving an open problem posed by Barwick and Ebert in \cite{BE}. In the same article, the authors explicitly describe the stabilizer in $P\Gamma L(3,q^2)$ of a BT-unital.

In this paper, we provide a classification up to projective equivalence in arbitrary-dimensional finite projective spaces. Our results show that, in this respect, BT quasi-Hermitian varieties behave analogously to BT unitals in the projective plane: they are all projectively equivalent. Then, we investigate the structure of BT quasi-Hermitian varieties in $\PG(3,q^2)$ and determine their full collineation group.

Finally, as a byproduct of our analysis, we establish the equivalence of certain minimal codes introduced in \cite{ALGS}.

Our long-term goal is to develop a group-theoretical characterization of BT quasi-Hermitian varieties, in order to distinguish them from all other quasi-Hermitian varieties in projective spaces of the same dimension and order. 
In this direction, Lemma~\ref{SylHigher},
which states that any Sylow $p$-subgroup of the collineation group of a BT quasi-Hermitian variety fixes a unique incident point-hyperplane pair of $PG(r,q^{2})$ (any $r>2$),
appears to be a promising starting point.

\section{Preliminaries}\label{prelim}

In \cite{A}, the first author introduced a new  family of  non-classical quasi-Hermitian varieties which are linked to the Buekenhout-Tits unitals. The construction is as follows.

Let consider the projective space $\PG(r,q^2)$ of order $q=2^e$, $e$ an odd number, the homogeneous coordinates of its points and denote by $$(X_0, X_1, \ldots, X_r).$$ Let $$\Pi_{\infty}:=\{(X_0, X_1,\ldots, X_r)\in \PG(r,q^2)| \, X_0=0 \}$$ be the hyperplane of the 'points at infinity'. Then $\AG(r,q^2)$ is the affine space  $\PG(r,q^2)\setminus \Pi_{\infty}$ with points having coordinates $(x_1,\ldots,x_r)$ where $x_i=X_i/X_0$ for $i=1,\ldots,r$.

The the trace and norm functions of $\GF{q^2}$ over $\GF{q}$ are
\[
\tr_{q^2/q}(x):=x+x^q,\qquad \mathrm{N}_{q^2/q}(x):=x^{q+1}\qquad \text{ for all }x\in\GF{q^2},
\]
respectively, and $\sigma:x\rightarrow x^{2^{\frac{e+1}{2}}}$ is an automorphism of $\GF{q}$, being $q=2^e$ with $e$ odd.

Let $\varepsilon\in \GF{q^2}\setminus \GF{q}$ be such that $\varepsilon^2+\varepsilon+\delta=0$ for some $\delta \in \GF{q}\setminus \{1\}$, $\tr_{q/2}(\delta)=1$. It is straightforward to see that   $\tr_{q^2/q}(\varepsilon)=1$. The pair $(1, \varepsilon)$ is a basis of $\GF{q^2}$ over $\GF{q}$, and hence $a=a_0+a_1\varepsilon$ with $a_0, a_1$ suitable elements in $\GF{q}$ for each $a \in \GF{q^2}$.

Let  $\mathcal{V}^r_\varepsilon$ be the variety with affine equation:
\begin{equation}\label{eq:V}
x_r^q+x_r=\Gamma_{\varepsilon}(x_1)+\ldots+\Gamma_{\varepsilon}(x_{r-1})\text{,}
\end{equation}
where $$\Gamma_{\varepsilon}(x)=[x+(x^q+x)\varepsilon]^{\sigma+2}+(x^q+x)^{\sigma}+(x^{2q}+x^2)\varepsilon+x^{q+1}+x^2$$
and consider the Hermitian cone
\begin{equation}\label{eq:Hinf}
\mathcal{\cH}^r_{\varepsilon,\infty}:=\{(0,X_1,\ldots,X_r)|\quad X_1^{q+1}+\ldots+X_{r-1}^{q+1}=0\}.
\end{equation}
Let us define

\begin{equation}\label{eq:H}
\mathcal{H}^r_\varepsilon:=(\mathcal{V}_{\varepsilon}^r\setminus \Pi_{\infty})\cup \mathcal{\cH}^r_{\varepsilon,\infty}.
\end{equation}
It is shown in \cite{A} that, $\mathcal{H}^r_\varepsilon$ is a non-classical quasi-Hermitian variety of $\PG(r,q^2)$ and a Buekenhout-Tits unital for $r=2$.
We call $\mathcal{H}^r_\varepsilon$ a \emph{Buekenhout-Tits} or \emph{BT quasi-Hermitian variety} of $\PG(r,q^2)$ for $r\geq 2$.


Let $M_1$ and $M_2$ be the number of hyperplanes of $PG(r,q^{2})$ intersecting
the BT-variety $\mathcal{H}^3_\varepsilon$ in $$N_1=\frac{\left( q^{r}-(-1)^{r}\right) \left(
q^{r-1}-(-1)^{r-1}\right) }{q^{2}-1}$$ or $$N_2=\frac{\left(
q^{r}-(-1)^{r}\right) \left( q^{r-1}-(-1)^{r-1}\right) }{q^{2}-1}%
+(-1)^{r-1}q^{r-1}$$ points, respectively. Since $\mathcal{H}^r_\varepsilon$ has the same size,  as well as the intersection numbers with the hyperplanes, of a non-degenerate Hermitian variety $\mathcal{H}(r,q^2)$, it follows that 
$$
M_{1}=\frac{\left( q^{r+1}-(-1)^{r+1}\right) \left( q^{r+1}+(-1)^{r+1}-\left(
q^{2}-1\right) ^{2}\right) }{\left( q^{2}-1\right) }
$$
and $$
M_{2}=\left(
q^{2}-1\right) \left( q^{r+1}-(-1)^{r+1}\right) \text{.} 
$$
 In particular, $M_2\equiv 1  \pmod{p}$. Therefore, any Sylow $p$%
-subgroup of $Aut(\mathcal{H}^3_{\varepsilon})$, the full automorphism group of $\mathcal{H}^3_{\varepsilon}$, preserves at least a hyperplane of $PG(r,q^{2})$ among the $X_2$-secant ones. 

As pointed out in \cite[Remark at pag. 36]%
{A}, the elementary abelian $p$-group of order $q^{r}$ defined by $$E=\left\langle \tau_{\gamma_{1},\gamma_{2}, \gamma_{3}} :\gamma _{i}\in \mathbb{F}_{q},i=1,...,r\right\rangle\text{,}$$ where $\tau_{\gamma_{1},\gamma_{2}, \gamma_{3}}$ is the collineation of $PGL_{4}(q^{2})$ represented by the matrix 
\[
\left( 
\begin{array}{cccccc}
1 & \gamma _{1}\varepsilon  & \gamma _{2}\varepsilon  & \cdots  & \gamma
_{r-1}\varepsilon  & \gamma _{r}+(\gamma _{1}+\cdots +\gamma _{r-1})^{\sigma
}\varepsilon  \\ 
0 & 1 & 0 & \cdots  & 0 & \gamma _{1}+\gamma _{1}\varepsilon  \\ 
0 & 0 & 1 & \cdots  & 0 & \gamma _{2}+\gamma _{2}\varepsilon  \\ 
\vdots  & \vdots  & \vdots  & \ddots  & \vdots  & \vdots  \\ 
0 & 0 & 0 & \cdots  & 1 & \gamma _{r-1}+\gamma _{r-1}\varepsilon  \\ 
0 & 0 & 0 & \cdots  & 0 & 1%
\end{array}%
\right)\text{,} 
\]%
is a subgroup $Aut(\mathcal{H}^r_{\varepsilon})$. It is easy to see that $\Pi _{\infty }:X_{0}=0$, and its point $P_{\infty }(0,0,0,\ldots,0,1)$, are the unique $E$-invariant hyperplane and point of $PG(r,q^{2})$, respectively. Moreover, the following holds:

\begin{lemma}\label{SylHigher}
Any Sylow $p$-subgroup of $Aut(\mathcal{H}^r_{\varepsilon})$ fixes a unique incident point-hyperplane pair of $PG(r,q^{2})$. 
\end{lemma}
\begin{proof}
Let $U$ be any Sylow $p$-subgroup of $Aut(\mathcal{H}^r_{\varepsilon})$ containing $E$ and let $\Pi _{U}$ be a $U$-invariant hyperplane secant to $\mathcal{H}^r_{\varepsilon}$ since $M_2\equiv 1  \pmod{p}$. Then $\Pi _{U}=\Pi _{\infty }$ since $\Pi _{\infty }$ is the unique $E$-invariant hyperplane of $PG(r,q^{2})$. Moreover, $U$ fixes at least one of the $\frac{q^{2(r+1)}-1}{q^{2}-1}$ points of $PG(r,q^{2})$, which necessarily coincide with $P_{\infty }$ since this one is the unique point of $PG(r,q^{2})$ fixed by $E$ with $E \leq U$. Therefore, $P_{\infty }(0,0,0,1)$ and $\Pi _{\infty }:X_{0}=0$ are the unique $U$-invariant
point and hyperplane of $PG(r,q^{2})$. The assertion now follows from fact that all Sylow $p$-subgroups of $Aut(\mathcal{H}^r_{\varepsilon})$ lies in a unique conjugacy class.   
\end{proof}

\section{Projective Equivalence of $\cH_{\varepsilon}^r$}
We are going to prove that the BT quasi-Hermitian varieties $\cH_{\varepsilon}^r$ in
$\PG(r, q^2)$, $r\geq 2$
 are projectively equivalent.

\begin{theorem} \label{teofi}
Let $\varepsilon _{1},\varepsilon _{2}\in \mathbb{F}_{q^{2}}$ be such that $%
\varepsilon _{i}^{2}+\varepsilon _{i}=\delta _{i}$ with $T(\delta _{i})=1$, $%
i=1,2$. Assume that $\alpha =\varepsilon _{2}-\varepsilon _{1}$, $B=(\frac{\alpha}{\delta_1})^{\sigma/2}$, $\rho =\left(
\frac{\delta _{1}}{\delta _{2}}\right) ^{\sigma /2+1}$.  Then, the projectivity
$\xi $ represented by the matrix of order $r+1$
\[\left(
\begin{array}{cccccc}
1 &  \ B \rho \frac{ \varepsilon _{2}}{\varepsilon _{1}} & \
B \rho \frac{ \varepsilon _{2}}{\varepsilon _{1}} &\cdots&  \ B\rho \frac{\varepsilon _{2}}{\varepsilon _{1}} & \ \ (r-1)\rho^2B^{\sigma+2}\frac{\varepsilon_2^{q+2
}}{\varepsilon_1^{q+1}} \\
0 & \frac{\rho \varepsilon _{2}}{\varepsilon _{1}} & 0& \cdots  & 0  & B \rho^{2}\frac{\varepsilon _{2}^{q+1}}{\varepsilon _{1}^{q+1}} \\
0 & 0 & \frac{\rho \varepsilon _{2}}{\varepsilon _{1}}& \cdots &0 &B \rho^{2}\frac{\varepsilon _{2}^{q+1}}{\varepsilon _{1}^{q+1}} \\
\vdots  & \vdots  & \vdots  & \ddots  & \vdots  & \vdots  \\
0&0&0&\cdots&\frac{\rho \varepsilon _{2}}{\varepsilon _{1}}&B \rho^{2}\frac{\varepsilon _{2}^{q+1}}{\varepsilon _{1}^{q+1}}\\
0 & 0 & 0 &\cdots&0& \frac{\rho ^{2}\varepsilon _{2}^{q+1}}{\varepsilon _{1}^{q+1}}%
\end{array}%
\right)
\]%
maps $\mathcal{H}_{\varepsilon_{1}}^r$ onto $\mathcal{H}_{\varepsilon _{2 }}^r$%
. In particular, $\mathcal{H}_{\varepsilon_{1}}^r$ and $\mathcal{H}_{\varepsilon
_{2 }}^r$ are equivalent.
\end{theorem}

\begin{proof}
Clearly, $\xi $ preserves $\Pi _{\infty }$. Now, let $P$ be any point of $%
\mathcal{H}_{\varepsilon _{1}}^r$. If $P\in \Pi _{\infty }\cap \mathcal{H}%
_{\varepsilon _{1}}^r$ then $P=(0,a_1,\ldots,a_{r-1} ,a_r )$ with $N(a_1)+\cdots+N(a_{r-1})=0$, and hence
\[
P^{\xi }=\left( 0,a_1\frac{\rho \varepsilon _{2}}{\varepsilon _{1}},\ldots,a_{r-1}\frac{\rho \varepsilon _{2}}{\varepsilon _{1}},\mu  \right),
\]%
 with $\mu \in \GF{q^2}$ lies in $\Pi _{\infty }\cap \mathcal{H}_{\varepsilon _{1}}^r=\Pi _{\infty
}\cap \mathcal{H}_{\varepsilon _{2}}^r$ since
$$ N \left(a_1\frac{\rho \varepsilon _{2}}{\varepsilon _{1}} \right)+\cdots+N\left(a_{r-1}\frac{\rho \varepsilon _{2}}{\varepsilon _{1}}\right)=N \left(\frac{\rho \varepsilon _{2}}{\varepsilon _{1}}\right)\left[ N(a_1)+\cdots+N(a_{r-1})\right]=0\text{.}$$

If $P\in \mathcal{H}_{\varepsilon _{1}}^r\setminus \Pi _{\infty }$, then $%
P=(1,a_1,\ldots,a_{r-1} ,a_r )$ with $a_{i}=a_{i,0}+a_{i,1}\varepsilon$ and (unique) $a_{i,0},a_{i,1} \in \GF{q}$ for each $i=1,...,r$. Hence,
\begin{equation}
a_{r,1}=a_{1,0}^{\sigma +2}+a_{1,0}a_{1,1}+a_{1,1}^{\sigma }+\ldots+a_{r-1,0}^{\sigma
+2}+a_{r-1,0}a_{r-1,1}+a_{r-1,1}^{\sigma }\text{,}
\label{agios}
\end{equation}%
and hence $P^{\xi }=(1,X_1,\ldots,X_{r-1},X_r)$ with
\[
X_i=\frac{\rho \varepsilon _{2}}{\varepsilon _{1}}\left( B+a_i\right), i=1,...,r-1, \text{ and } X_r=\frac{%
\rho ^{2}\varepsilon _{2}^{q+1}}{\varepsilon _{1}^{q+1}}\left(
B(a_1+\ldots+a_{r-1}) +a_r\right)+(r-1)B^{\sigma+2}\varepsilon_2 \text{.}
\]
 Since $\delta_1=\varepsilon_1^{q+1}$ we get $1/\varepsilon_1=(1+\varepsilon_1)/\delta_1$ and hence
 $\varepsilon_2/\varepsilon_1=\delta_2/\delta_1+(\alpha/\delta_1) \varepsilon_2$   Thus, 
 $$X_i= \frac{\rho \delta_2}{\delta_1}(B+a_{i,0})+\rho\left[\frac{\alpha}{\delta_1}(B+a_{i,0})+a_{i,1}\right]\varepsilon_2,$$ for all
$i=1,...,r-1$.

Now, $P^{\xi }\in \mathcal{H}_{\varepsilon _{2}}^r$ if and only if \[T(X_r)=
\Gamma _{\varepsilon _{2}}(X_1)+\ldots+\Gamma _{\varepsilon _{2}}(X_{r-1}),\] which is
equivalent to
\[
T_{X_r}=Re_{X_1}^{\sigma +2}+Re_{X_1}T_{X_1}+T_{X_1}^{\sigma }+\ldots +Re_{X_{r-1}}^{\sigma +2}+Re_{X_{r-1}}T_{X_{r-1}}+T_{X_{r-1}}^{\sigma }\text{,}
\]%
where
$
Re_{X_i} =\frac{\rho \delta_2}{\delta_1}(B+a_{i,0})$, 
$T_{X_i}=T(X_i) =\rho\left[\frac{\alpha}{\delta_1}(B+a_{i,0})+a_{i,1}\right]
$ for all
$i=1,...,r-1$, and 
$$T_{X_r}=
T(X_r)=\frac{\delta _{2}}{\delta _{1}}\rho ^{2}\left[
a_{r,1}+B\left( a_{1,1}+ \ldots +a_{r-1,1})+(r-1)B^{\sigma+2} \right)
\right]$$
Indeed, as $\rho^{(\sigma-2)}=\delta_2/\delta_1$, that is $\rho^{\sigma}=\rho^2\frac{\delta_2}{\delta_1}$, we get
\[T(X_r)=\rho ^{\sigma }\left[a_{1,0}^{\sigma +2}+a_{1,0}a_{1,1}+a_{1,1}^{\sigma }+\ldots+a_{r-1,0}^{\sigma
+2}+a_{r-1,0}a_{r-1,1}+a_{r-1,1}^{\sigma }\right]+ \]
\[+\rho ^{\sigma }B\left(
a_{1,1}+ \ldots+a_{r-1,1}\right) %
+\rho^{\sigma}(r-1)B^{\sigma+2}\]%
by (\ref{agios}). Moreover, 
\begin{equation}
Re_{X_i}^{\sigma+2}=(\rho \delta_2/\delta_1)^{\sigma+2}(a_{i,0}^{\sigma+2}+a_{i,0}^2B^{\sigma}+a_{i,0}^{\sigma}B^2+B^{\sigma+2}),
\end{equation}

\begin{equation}
Re_{X_i} T_{X_i}=\alpha(\rho^2 \delta_2/\delta_1^2)(a_{i,0}^2+B^2)+(\rho^2 \delta_2/\delta_1)(a_{i,0}a_{i,1}+a_{i,1}B),
\end{equation}

\begin{equation}
 T_{X_i}^{\sigma}=\rho ^{\sigma}(\alpha/\delta_1)^{\sigma}(a_{i,0}^{\sigma}+B^{\sigma})+\rho^{\sigma} a_{i,1}^{\sigma},
 \end{equation}
 for $i=1,\ldots r-1$.

Thus, given that $(\rho \delta_2 / \delta_1)^{\sigma+2}=\rho^2 \delta_2 / \delta_1=\rho^{\sigma}$, $B^{\sigma}=\alpha/\delta_1$ and  $\sigma^2=2$, we obtain

\[(\rho \delta_2 / \delta_1)^{\sigma+2}(a_{i,0}^{\sigma}B^2+a_{i,0}^2B^{\sigma})+
\rho^2\delta_2 \alpha(a_{i,0}^2+B^2)/\delta_1^2+\rho^{\sigma}\alpha^{\sigma}(a_{i,0}^{\sigma}+B^{\sigma})/\delta_1^{\sigma}=\]
\[\rho^{\sigma}\left[a_{i,0}^{\sigma}\frac{\alpha^{\sigma}}{\delta_1^{\sigma}}+a_{i,0}^2\frac{\alpha}{\delta_1}+a_{i,0}^2\frac{\alpha}{\delta_1} +\frac{\alpha^{\sigma+1}}{\delta_1^{\sigma+1}}+a_{i,0}^{\sigma}\frac{\alpha^{\sigma}}{\delta_1^{\sigma}}+\frac{\alpha^{\sigma+1}}{\delta_1^{\sigma+1}}\right]=0.\]
This implies

$$\Gamma _{\varepsilon _{2}}(X_i)=Re_{X_i}^{\sigma+2}
+Re_{X_i}T_{X_i}+T_{X_i}^{\sigma }=
\rho ^{\sigma }( B^{\sigma
+2}+a_{i,0}a_{i,1}+a_{i,0}^{\sigma +2}+a_{i,1}^{\sigma }+Ba_{i,1})$$ for all $i=1,\ldots,r-1$.
Therefore,
\[
\Gamma _{\varepsilon _{2}}(X_1)+\ldots+\Gamma _{\varepsilon _{2}}(X_{r-1})=
T_{X_r}
\]
hence, $(\mathcal{H}_{\varepsilon_{1} }^r)^{\xi }=\mathcal{H}_{\varepsilon _{2 }}^r$%
, which is the assertion.
\end{proof}

\section{The $3$-dimensional case}
\label{minlinq}
This section is devoted to the analysis of the BT quasi-Hermitian varieties of $PG(3,q^{2})$. In the first part, we focus on some geometric properties of $\cH_{\varepsilon}^3$; in the second part, we use these information to completely determine the collineation group of $\cH_{\varepsilon}^3$.

\subsection{Some geometric properties of $\cH_{\varepsilon}^3$}
Consider the variety $\mathcal{V}^3_{\varepsilon} \subset \PG(3,q^2)$ with affine equation
\begin{equation}\label{eq3:V}
x_3^q+x_3=\Gamma_{\varepsilon}(x_1)+\Gamma_{\varepsilon}(x_2).
\end{equation}
We are going to study the properties  of the lines contained in $\mathcal{V}_{\varepsilon}^3$ and deduce some related properties of the lines contained in the associated BT quasi-Hermitian variety. 
 Let $\Pi_{\infty}$ be the plane at infinity of $\PG(3,q^2)$.
Set $\mathcal{V}_{\varepsilon,\infty}^3:=\mathcal{V}_{\varepsilon}^3 \cap \Pi_{\infty}$.

The following lemma can be found in \cite{ALGS}[Lemma 3.1].
\begin{lemma}\label{lemma:intersection}
 
\[
\mathcal{V}^3_{\varepsilon,\infty}=\begin{cases}
 \ell_{\infty} & \text{ if } e\equiv 1\; mod\, 4\\
    \ell_{\infty}\cup \ell_1 \cup \ell_2 & \text{ if }
    e\equiv 3 \;mod \,4,
\end{cases}
\]
where $\ell_{\infty}:X_0= X_1+X_2=0$ and $\ell_i:X_0=X_1+w_i X_2=0$ for $i=1,2$ and
$w_1$,  $w_2$  are the $(2^\frac{e-1}{2}+1)$-th  roots of unity in $\GF{q^2}$ different from $1$.
\end{lemma}

  

  We refer to 
   the points in $\ell_{\infty}$ as
  $P_{\infty}=(0,0,0,1)$ and $L^n_{\infty}=(0,1,1,n)$ with
  $n\in\GF{q^2}$.

\begin{theorem} \label{th30} Let $\mathcal{V}^3_{\varepsilon}$ be the surface of $\PG(3,q^2)$ of
  equation \eqref{eq3:V}. Then the following hold:
  \begin{enumerate}
  \item  any line of $\PG(3,q^2)$ contained in $\mathcal{V}^3_{\varepsilon}$ contains at least one of the points 
   $L^{\varepsilon^{q}\alpha}_{\infty}=(0,1,1,\varepsilon ^{q}\alpha)$ with $\alpha \in \GF{q}$; 
  \item through any point $L^{\varepsilon^{q}\alpha}_{\infty}$,  there
    are exactly $q+1$ coplanar lines contained in $\mathcal{V}^3_{\varepsilon}$, one of these being $\ell_{\infty}$. Each such set of $q+1$ lines is an Hermitian cone $\Pi_0\mathcal{H}(1,q^2)$;
  \item  either $\ell_{\infty}$ or $\ell_{\infty}, \ell_1,\ell_2$ are the unique lines contained in
    $\mathcal{V}^3_{\varepsilon}$ through $P_{\infty}=(0,0,0,1)$ according as $e \equiv 1 \pmod 4$  or $e \equiv 3 \pmod 4$, respectively.
  \end{enumerate}
\end{theorem}

\begin{proof}
Assertion (1) is obvious if  $\ell\in\Pi_{\infty}$. Hence, let $\ell\in\PG(3,q^2)\setminus \Pi_{\infty}$ contained in $\mathcal{V}_{\varepsilon}^3$. Then $\ell$ contains one of the points in   $\mathcal{V}^3_{\varepsilon,\infty}$. Now, let $P(1,a,b,c)\in \mathcal{V}_{\varepsilon}^3\cap\AG(3,q^2)$. So, \[c^q+c=\Gamma_{\varepsilon}(a)+\Gamma_{\varepsilon}(b).\] Suppose that $\ell$ passes through $P$ and $L_{\infty}^n$, that is $\ell$ has the following affine
  parametric equations:
  \[
    \begin{cases}
      x_1=a+t \\
      x_2=b+t \\
      x_3=c+nt
    \end{cases}
  \]
  for $t\in\GF{q^2}$ and $n$ a given element in $\GF{q^2}$.
  This line is contained in $\mathcal{V}_{\varepsilon}^3$ if and only if 
\[[(a+b)+T(a+b)\varepsilon]^2[t+T(t)\varepsilon]^{\sigma}+[(a+b)+T(a+b)\varepsilon]^{\sigma}[t+T(t)\varepsilon]^{2}+T((a+b)t^q)=T(nt)\]
is fulfilled for each $t\in\GF{q^2}$. This yields either $n=0$, $a=b$ and $c \in \GF{q}$, or $n\neq 0$, $a_0=b_0$, $n=a^q+b^q=(a_1+b_1)(1+\varepsilon)$ and $c^q+c=\Gamma_{\varepsilon}(a)+\Gamma_{\varepsilon}(b)$. Therefore, there are $q$ points on the  line $\ell_{\infty}$  which are contained in $q+1$ lines of $\mathcal{V}_{\varepsilon}^3$.  
 The $q$ affine lines through $L_{\infty}^0$ have affine equations
  \[
    \begin{cases}
      x_1=x_2\\
      x_3=\lambda \\
         \end{cases}
  \]
 where $\lambda \in \GF{q}$.
  These $q$ lines, together with the line $\ell_{\infty}$, form a Hermitian cone with vertex  $L_{\infty}^0$ and basis the Baer subline $\{(0,0,0,1)\} \cup \{(1,a,a,\lambda) |\lambda \in \GF{q} \}$ in the plane $x_1+x_2=0$.
  
Next,  consider the affine lines through the point $L_{\infty}^{ \alpha+\alpha \varepsilon}$, $\alpha \in \GF{q}^*$, that are contained in $\mathcal{H}^3_{\varepsilon}$, that is the lines with affine
equations 
\[
    \begin{cases}
      x_1+x_2=a+b \\
    n x_1+ x_3= c+na \\
         \end{cases}
  \]
where $a=a_0+a_1\varepsilon$, $b=a_0+\varepsilon b_1$, $a_1+b_1= \alpha$ and  $c^q+c=\Gamma_{\varepsilon}(a)+\Gamma_{\varepsilon}(b)$.
These lines, together with $\ell_{\infty}$, form a Hermitian cone with vertex $L_{\infty}^n$  and basis the baer subline $\{(0,0,0,1)\} \cup \{(1,a,b,c+\lambda) |\lambda \in \GF{q} \}$ in the plane $x_1+x_2+ \alpha \varepsilon x_0=0$. This proves (1) and (2).

If $e \equiv 3 \pmod 4$ then $\ell$ can also pass through $P$ and a point  $M_{\alpha}^n(0,1,\alpha,n)$ of $\ell_i$ and hence $\ell$ has the following affine  parametric equations:
  \[
    \begin{cases}
      x_1=a+t \\
      x_2=b+\alpha t \\
      x_3=c+nt
    \end{cases}  \]
  for $t\in\GF{q^2}$, $\alpha=w_i$ and $n$ a given element in $\GF{q^2}$.
Thus, $\ell$ is contained in $\mathcal{V}^3_{\varepsilon}$ if and only if 
\begin{equation} \label{eq8}
c^q+c+n^qt^q+n t=\Gamma_{\varepsilon}(a+t)+\Gamma_{\varepsilon}(b+\alpha t)
\end{equation}
for each $t\in \GF{q^2}$.
Write $a=a_0+\varepsilon a_1$, $b=b_0+\varepsilon b_1$, $\alpha=\alpha_0+\varepsilon \alpha_1$ and $t=t_0+\varepsilon t_1$. Then $a+t=a_0+t_0 +\varepsilon (a_1+t_1)$ and $b+\alpha t=(b_0+\alpha_0t_0+\delta \alpha_1t_1)+ \varepsilon (b_1+\alpha_0t_1+\alpha_1(t_0+t_1))$.
We obtain 
\begin{align}
\Gamma_{\varepsilon}(a+t)&= \Gamma_{\varepsilon}(a)+\Gamma_{\varepsilon}(t)+a_0^2t_0^{\sigma}+a_0^{\sigma}t_0^2+a_0t_1+a_1t_0  \label{Dekaexi}\\
\Gamma_{\varepsilon}(b+\alpha t)&= \Gamma_{\varepsilon}(b)+\Gamma_{\varepsilon}(\alpha t)+b_0^2 (\alpha_0t_0+\delta \alpha_1t_1)^{\sigma}+b_0^{\sigma}(\alpha_0t_0+\delta \alpha_1t_1)^2+  \label{Dekaepta}\\
&\phantom{=} b_0(\alpha_0t_1+\alpha_1(t_0+t_1))+b_1(\alpha_0t_0+\delta \alpha_1t_1). \nonumber
\end{align}
Now, substituting \eqref{Dekaexi} and \eqref{Dekaepta} in \eqref{eq8}, and bearing in mind that $\Gamma_{\varepsilon}(a)=a_0^{\sigma+2}+a_0a_1+a_1^{\sigma}$, it is straightforward to check that $\alpha=1$ is the unique admissible value for which \eqref{eq8} is fulfilled for any $t\in\GF{q^2}$, a contradiction.

Finally, suppose that $\ell$ passes through $P$ and $P_{\infty}$, hence its 
  parametric equations are:
\[
    \ell:\begin{cases}
      x_1=a\\
      x_2=b \\
      x_3=c+t
    \end{cases}
  \]
 with $t\in\GF{q^2}$. In this case, $\ell$ meets $\mathcal{V}_{\varepsilon}^3$ in $q$ affine points and assertion (3) follows.
 
 \end{proof}

\begin{theorem} \label{th32} Let $\mathcal{H}^3_{\varepsilon}$ be the BT quasi-Hermitian variety of $\PG(3,q^2)$ described by \eqref{eq:H}. Then the following statements hold: \begin{enumerate} 
\item each affine line contained in $\mathcal{H}^3_{\varepsilon}$ intersects $\ell_{\infty}$ at one among the points
$L^{\varepsilon^{q}\alpha}_{\infty} = (0,1,1,\varepsilon^{q}\alpha)$, where $\alpha \in \GF{q}$;

\item through each point $L^{\varepsilon^{q}\alpha}_{\infty}$, there exist exactly $q+1$ coplanar lines contained in $\mathcal{H}^{3}_{\varepsilon}$, one of which is $\ell_{\infty}$. Each such set of $q+1$ lines forms a Hermitian cone $\Pi_0\mathcal{H}(1,q^2)$; \end{enumerate} \end{theorem}
\begin{proof}

We observe that $\mathcal{H}^3_\varepsilon \cap \AG(3,q^2) = \mathcal{V}^3_{\varepsilon} \cap \AG(3,q^2)$ and that $\ell_{\infty} \subset \mathcal{H}^r_{\varepsilon, \infty}$.
Therefore, the result directly follows from Theorem \ref{th30}.
\end{proof}

\subsection{The automorphism group of $\mathcal{H}_{\varepsilon}^3$}
In this section we determine $Aut(\mathcal{H}^{3}_{\varepsilon})$, the full automorphism group of $\mathcal{H}_{\varepsilon}^3$. The symbol $A$ will denote the stabilizer in $Aut(\mathcal{H}_{\varepsilon}^3) \cap \PGL{4}{q^2}$ of $\Pi_{\infty}$.

\begin{lemma}\label{Syl1}
$A=E\left\langle \vartheta \right\rangle$, where $\theta$ is the collineation with associated  matrix
\begin{equation}
\left( 
\begin{array}{cccc}
1 & 0 & 0 & 0 
\\ 
0 & 0 & 1 & 0 \\ 
0 & 1 & 0 & 0 \\ 
0 & 0 & 0 & 1%
\end{array}%
\right)
\end{equation}
In particular, $|A:E|=2$.
\end{lemma}

\begin{proof}
Let $\varphi \in A$, then $\varphi$ is represented by the (non-singular) matrix
\[M=\begin{pmatrix}
		1 & a & b & c \\
		0 & e+h+g&e &f         \\
		0 &  g&h&i       \\
		0 &0&0 & l
	\end{pmatrix}.
	\]
 since $\varphi$ preserves $\Pi_{\infty}$, $\ell_{\infty}$ and the point $P_{\infty}$.
As $P_{w}=(1,0,0,w) \in \mathcal{H}_{\varepsilon}^3$ with $w \in \GF{q}$, it follows that $P^{\varphi}_{w}=(1,a_0+a_1\varepsilon,b_0+b_1\varepsilon,c_0+wl_0+(c_1w+l_1\varepsilon)) \in  \mathcal{H}_{\varepsilon}^3$, and hence
\begin{equation}\label{aut1}
c_1+wl_1=a_0^{\sigma+2}+a_1^{\sigma}+a_0a_1+b_0^{\sigma+2}+b_1^{\sigma}+b_0b_1
\end{equation}  
Thus $l_{1}=0$, and so $l\in \GF{q}^{\ast}$. 

Each $P_{s,z}=(1,0,s,z)$ with $s \in  \GF{q}$ and $z^q+z=s^{\sigma+2}$ is a point of $\mathcal{H}_{\varepsilon
}^3$, then $P_{s,z}^{\varphi}$ belongs to $\mathcal{H}_{\varepsilon}^3$. Now, taking into account \eqref{aut1}, $P^{\varphi}=(1,a+sg, b+sh, c+is+lz)\in \mathcal{H}_{\varepsilon}^{3}$, one has
\[
c_1+si_1+ls^{\sigma+2}=a_0^{\sigma+2}+a_1^{\sigma}+a_0a_1+b_0^{\sigma+2}+b_1^{\sigma}+b_0 b_1+\]
\[(sg_0)^{\sigma+2}+a_0^{\sigma}s^2g_0^2+a_0^{2}s^{\sigma}g_0^{\sigma}+s^{\sigma}g_1^{\sigma}+a_0sg_1+a_1sg_0+s^2g_0g_1+\]
\[(sh_0)^{\sigma+2}+b_0^{\sigma}s^2h_0^2+b_0^{2}s^{\sigma}h_0^{\sigma}+s^{\sigma}h_1^{\sigma}+b_0sh_1+b_1sh_0+s^2h_0h_1 \]
for each $s \in \GF{q}$. Therefore, we obtain
\begin{equation}\label{r1}
        i_1=a_0g_1+g_0a_1+b_1h_0+b_0h_1
        \end{equation}
   \begin{equation}\label{r2}   
   l=g_0^{\sigma+2}+h_0^{\sigma+2} 
   \end{equation}
              \begin{equation}\label{r3}  
            a_0^2g_0^{\sigma}+b_0^2h_0^{\sigma}+h_1^{\sigma}+g_1^{\sigma}=0
            \end{equation}
        \begin{equation}\label{r4}
      a_0^{\sigma}g_0^{2}+b_0^{\sigma}h_0^{2}+g_0g_1+h_0h_1=0
    \end{equation}
 Similarly, $Q_{t,z}=(1,0,t \varepsilon,z)$ with $t \in  \GF{q}$ and $z^q+z=t^{\sigma}$ is a point of $\mathcal{H}_{\varepsilon
}^{3}$ then $Q_{t,z}^{\varphi}=(1,a+\varepsilon t g,b+\varepsilon th, c+i\varepsilon t+lz) \in \mathcal{H}_{\varepsilon}^{3}$,
leading to 
\begin{equation}\label{r5}
          i_0+i_1=a_0g_1+g_0a_0+\delta a_1g_1+\delta b_1h_1+b_0h_0+b_0h_1
          \end{equation}
   \begin{equation}\label{r6}       
      g_1^{\sigma+2}=h_1^{\sigma+2} \end{equation}
      \begin{equation}\label{r7}
      a_0^2 \delta^{\sigma} g_1^{\sigma}+b_0^2 \delta^{\sigma}h_1^{\sigma}+g_1^{\sigma}+g_0^{\sigma}+h_1^{\sigma}+h_0^{\sigma}=l
      \end{equation}
   \begin{equation}\label{r8}    
      a_0^{\sigma}  \delta^{2} g_1^{2}+b_0^{\sigma} \delta^{2} h_1^{2}+ \delta g_0g_1+ \delta h_0h_1+ \delta g_1^2+ \delta h_1^2=0
     \end{equation}

From \eqref{r6}, we obtain $h_1=g_1$ since $\sigma+2 \in Aut(\mathbb{F}^{\ast}_{q})$.

We are going to prove $f_{1}=g_{1}=0$. Thus, assume $g_{1},f_{1} \neq 0$. 
Note that $(a_0,b_0)\neq (0,0)$ and $a_0\neq b_0$, otherwise $f_0= g_0$ and $l=0$ which is impossible.  
Assume $a_0\neq 0$. From \eqref{r3}, we obtain $g_0^{\sigma}=\frac{b_0^2}{a_0^2}h_0^{\sigma}$, which substituted in \eqref{r4} leads to $h_0^{\sigma}(1+\frac{b_0^2}{a_0^2})(b_0^2h_0^{\sigma}+g_1^{\sigma})=0$, that is to say $h_0^{\sigma}=g_1^{\sigma}/b_0^{2}$. Hence, $h_0=g_1/b_0^{\sigma}$ and $g_0=g_1/a_0^{\sigma}$.

From \eqref{r8} we get $a_0^{\sigma}  \delta g_1+b_0^{\sigma} \delta g_1+h_0+g_0=0$, hence \[\delta g_1(a_0^{\sigma}+b_0^{\sigma})+g_1(1/b_0^{\sigma}+1/a_0^{\sigma})=0.\] Therefore, $\delta=1/(a_0^{\sigma}b_0^{\sigma})$, and hence $l=0$ in \eqref{r7} gives $l=0$, a contradiction. Thus, $h_1=g_1=0$. Moreover, we have  $g_0^2 + h_0^2 +1=0$, and hence $g_0=h_0+1$, from \eqref{r2} and \eqref{r7}.

Now, we are going to prove that $a_0=b_0=0$. Assume $a_0\neq 0$. From $g_1=h_1=0$ and equation  \eqref{r3} we have $a_0^2g_0^{\sigma}+b_0^2 h_0^{\sigma}=0$ and  since $\sigma^2=2$ we obtain $g_0=(b_0/a_0)^{\sigma} h_0$. From 
\eqref{r4} we get $h_0^2 b_0^{\sigma}(a_0^{\sigma}+b_0^{\sigma})=0$, namely $h_0=0$ or  $a_0=b_0$ or $b_0=0$.
If $h_0=0$ then $g_0=0$, impossible. If $a_0=b_0$ then again $g_0=h_0$, which is excluded.
Then $b_0=0$, $g_0=0$ and $h_0=1$. In particular $g=0$ and $d \neq 0$.
By considering the points $P_{s,z}^{\prime}=(1,s,0,z)$ and $Q_{t,z}^{\prime}$ for  $s,t  \in \GF{q}$, and repeating the same argument as above, we get $d_1=e_1=0$ and $a_0^2d_0^{\sigma}+b_0^2e_0^2=a_0d_0^2=0$ which gives a contradiction. Thus $a_0=0$. If $b_0 \neq 0$ this time $h=0$ and  $e\neq 0$. Again by considering the points $P_{s,z}^{\prime}$ and $Q_{t,z}^{\prime}=(1,\varepsilon t, 0, z)$ in $\cH_{\varepsilon}^3$, and repeating the above argument, we obtain $a_0^2d_0^{\sigma}+b_0^2e_0^2=b_0e_0^2=0$ leading to a contradiction. Hence $a_0=b_0=0$, $e_1=d_1=0$, $d_0=e_0+1$, $i_0=i_1=(a_1+b_1)h_0+a_1$, $f_0=f_1=(a_1+b_1)e_0+a_1$. Furthermore, $e_0=h_0+1$ since $\varphi$ preserves the set of points $R=(0,1,t,z)$ with $t^{q+1}=1$. Thus, $\varphi$ is represented by
\[
M=\left( 
\begin{array}{cccc}
1 & \gamma _{1}\varepsilon  & \gamma _{2}\varepsilon  & \gamma _{3}+\left(
\gamma _{1}^{\sigma }+\gamma _{2}^{\sigma }\right) \varepsilon  \\ 
0 & h & h+1 & \left( \left( \gamma _{1}+\gamma _{2}\right) h+\gamma
_{2}\right) \varepsilon ^{q} \\ 
0 & h+1 & h & \left( \left( \gamma _{1}+\gamma _{2}\right) h+\gamma
_{1}\right) \varepsilon ^{q} \\ 
0 & 0 & 0 & 1%
\end{array}
\right)\text{.}
\]
Now, set $\psi=\varphi\tau_{\gamma _{1},\gamma _{2},\gamma _{3}}$. Then $\psi$, which is represented by the matrix
\[
N =\left( 
\begin{array}{cccc}
1 & 0 & 0 & \varepsilon ^{q+1}\left( \gamma _{1}^{2}+\gamma _{2}^{2}\right) 
\\ 
0 & h & h+1 & 0 \\ 
0 & h+1 & h & 0 \\ 
0 & 0 & 0 & 1%
\end{array}%
\right)\text{,} 
\]
must preserve $\cH_{\varepsilon}^3$. Hence, $P^{\psi }=(1,hx+y(h+1),x(h+1)+hy,z+\varepsilon ^{q+1}\left( \gamma
_{1}^{2}+\gamma _{2}^{2}\right) )$ lies in $\mathcal{H}_{\varepsilon }^3$ whenever $P=(1,x,y,z)$ does it. This means
\begin{eqnarray*}
T(z) &=&\Gamma (y+h(x+y))+\Gamma (x+h(x+y))\\
T(z )&=&\Gamma (x)+\Gamma (y)\text{,}
\end{eqnarray*}%
which imply 
\begin{equation}
\Gamma (x)+\Gamma (y)=\Gamma (y+h(x+y))+\Gamma (x+h(x+y))\text{.}  \label{moimoi}
\end{equation}
Choosing, $x=1$, $y=0$ and $z \in \mathbb{F}_{q}$, \eqref{moimoi} implies $1=\Gamma (h)+\Gamma (1+h)=h^{\sigma+2}+(1+h)^{\sigma+2}=1+h^{2}+h^{\sigma}$, and hence $h^{\sigma }+h^{2}=0$. Then $%
\left( h+h^{\sigma }\right) ^{\sigma }=0$ since $\sigma ^{2}=2$, and hence $%
h^{\sigma }=h$. Therefore $h+h^{2}=0$, and hence $%
h=0,1$. Thus, either $\psi=\tau_{0,0,\gamma _{1}^{2}+\gamma _{2}^{2}}$ or $\psi=\tau_{0,0,\gamma _{1}^{2}+\gamma _{2}^{2}}\vartheta$, and hence $\psi \in E\left\langle \vartheta \right\rangle$. So $\varphi=\psi\tau_{\gamma _{1},\gamma _{2},\gamma _{3}} \in E \left\langle \vartheta \right\rangle$. On the other hand, $E\left\langle \vartheta \right\rangle \leq A$ and hence the assertion follows.
\end{proof}

\begin{corollary}\label{C1}
If $E$ is not normal in $A$, then there is $\varphi \in A$ such that $E^{\varphi }\neq E$ and $\Pi _{\infty }^{\varphi }\neq \Pi _{\infty }$.
\end{corollary}

\begin{proof}
Suppose that $E$ is not normal in $A$. Then there is $\varphi \in A$ such
that $E^{\varphi }\neq E$. Suppose that $\Pi _{\infty }^{\varphi }=\Pi
_{\infty }$, then $\left\langle E,E^{\varphi }\right\rangle $ induces a
group of elations of $\Pi _{\infty }$ having center $P_{\infty }$. Then $\left\langle E,E^{\varphi }\right\rangle =E$, and hence $E^{\varphi }= E$, and we reach a contradiction.
\end{proof}

\begin{theorem}\label{Jova} The  following hold:
\begin{enumerate}
    \item $Aut(\mathcal{H}_{\varepsilon }^3) \cap PGL_{4}(q^{2})=E\left\langle\vartheta \right\rangle$ is a group of order $2q^{3}$;
    \item $Aut(\mathcal{H}_{\varepsilon }^3)$ preserves the triple $(P_{\infty},\ell_{\infty},\Pi_{\infty})$.
\end{enumerate}
\end{theorem}

\begin{proof}
Suppose that $E$ is not normal in $A=Aut(\mathcal{H}_{\varepsilon}^3)\cap
PGL_{4}(q^{2})$. Then there is $\psi \in A$ such that $E^{\psi}\neq E$.
Then $\Pi _{\infty }^{\psi }\neq \Pi _{\infty }$ by Corollary \ref{C1} and hence $%
\Pi _{\infty }^{\psi }=$ $\Pi _{n_{0}}$, where $\Pi _{n_{0}}$ is the
hyperplane containing the hermitian $\mathcal{C}_{n_{0}}$ with apex $%
L^{n_{0}}_{\infty}=(0,1,1,n_{0})$, $n_{0}=\varepsilon ^{q}\alpha _{0}$ and $%
\alpha _{0}\in \mathbb{F}_{q}$. Then $\left\vert E^{A}\right\vert \geq q+1$
since there are precisely $q+1$ such Hermitian cones contained in $\mathcal{H%
}_{\varepsilon }^{3}$ by Theorem \ref{th32}(2), $\Pi _{\infty }^{\psi }=L_{n_{0}}$ and $E$ permutes transitively the points in the set $\left\{ L^{\alpha}_{\infty}: \alpha\in 
\mathbb{F}_{q}\right\}$. 

Let $E^{\ast }=\left\{ \varphi _{\gamma _{1},\gamma _{1},\gamma _{3}}\in
E:\gamma _{1},\gamma _{3}\in \mathbb{F}_{q}\right\} $ is a normal subgroup
of $E$ of order $q^{2}$ preserving both $\Pi _{\infty }$ and $\mathcal{C}%
_{\infty }$ as well as $\Pi _{n}$ and $\mathcal{C}_{n}$. In particular, $E^{\ast }$ a subgroup of $E^{\psi }$ fixing $\ell
_{\infty }$ pointwise. On the other hand, the subgroup $H$ of $E^{\psi }$ fixing $\ell
_{\infty }$ pointwise has index $q$ since the order of $E$ is $q^{3}$. Thus $E^{\ast}=H$, and hence $E^{\ast }$ the subgroup of $E^{\psi }$ fixing $\ell
_{\infty }$ pointwise. 

The group $%
K=\left\{ \varphi _{0,0,\gamma _{3}}\in E:\gamma _{3}\in \mathbb{F}%
_{q}\right\} \leq E^{\ast }$ is the kernel of the action of $E$ on $\Pi _{\infty }$. Therefore $K^{\psi }$ fixes $\mathcal{C}%
_{n_{0}}$ pointwise, and hence $K^{\psi }$ fixes $\mathcal{C}%
_{n_{0}}\setminus \ell _{\infty }$ pointwise \ with  $\mathcal{C}%
_{n_{0}}\setminus \ell _{\infty }\subseteq \mathcal{H}_{\varepsilon
}^3\setminus \Pi _{\infty }$. On the other hand, $K^{\psi }$ fixes $\ell
_{\infty }$ pointwise, and so $K^{\psi }\leq E^{\ast }$ since $E^{\ast }$ is
the subgroup of $E^{\psi }$ fixing $\ell _{\infty }$ pointwise. Then $%
E^{\ast }$, and hence $K^{\psi }$, acts semiregularly on $\mathcal{H}%
_{\varepsilon }\setminus \Pi _{\infty }$, which is not the case. Thus, $E$ is a normal subgroup of $Aut(\mathcal{H}_{\varepsilon }^3)$.

Now, let $W$ be any Sylow $2$-subgroup of $Aut(\mathcal{H}_{\varepsilon }^{3})$. Then $E \leq W$ since $E$ is a normal subgroup of $Aut(\mathcal{H}_{\varepsilon }^{3})$, and hence $W$ fixes the incident point-hyperplane pair $(P_{\infty},\Pi_{\infty})$ by Lemma \ref{SylHigher}. Thus, $Aut(\mathcal{H}_{\varepsilon }^3)$ fixes $(P_{\infty},\Pi_{\infty})$, and hence $Aut(\mathcal{H}_{\varepsilon }^3) \cap PGL_{4}(q^{2})=E\left\langle\vartheta \right\rangle=E\left\langle\vartheta \right\rangle$ by Lemma \ref{Syl1}.


\end{proof}

\begin{theorem}\label{Avtos}
    $Aut(\mathcal{H}_{\varepsilon}^3)=E\left\langle \vartheta, \phi \right\rangle$, with
    \[\phi: (x_0, x_1,x_2,x_3) \rightarrow (x_0^2,x_1^2,x_2^2,x_3^2) \left(
\begin{array}{cccc}
1 & 1 & 1 & 0 \\
0 & \delta^{\frac{\sigma}{2}} \varepsilon^q & 0 & \delta^{\frac{\sigma}{2}}\varepsilon^q\\
0 & 0 & \delta^{\frac{\sigma}{2}}\varepsilon^q   & \delta^{\frac{\sigma}{2}}\varepsilon^q \\
0 & 0 & 0 & \delta^{\sigma+1} %
\end{array}\right)\text{,}\]
is a group of order $4eq^{3}$.
\end{theorem}
\begin{proof}

First, we prove that $\phi$ leaves invariant the set of affine points of $\mathcal{H}^3_\varepsilon$. To this end, consider an affine point $P(1,x,y,z) \in \mathcal{H}^3_\varepsilon$. Then, $z_1=x_0^{\sigma+2}+x_0x_1+x_1^{\sigma}+y_0^{\sigma}+y_0^{\sigma+2}
+y_0y_1+y_1^{\sigma}$ and in particular
\begin{equation}\label{rot1}
z_1^2=x_0^{2(\sigma+2)}+x_0^2x_1^2+x_1^{2\sigma}+y_0^{2\sigma}+y_0^{2(\sigma+2)}
+y_0^2y_1^2+y_1^{2\sigma}
\end{equation}
We also observe that 
\begin{equation}\label{rot2}
\delta^{\sigma+1}z^2=\delta^{\sigma+1}
z_0^{2}+ \delta^{\sigma+2}z_1^2+ \varepsilon \delta^{\sigma+1}z_1^2
\end{equation}
and
\begin{equation}\label{rot3}
1+\delta^{\sigma/2}\varepsilon^q x^2=1+\delta^{\sigma/2}(x_0^{2})+\delta^{\sigma/2} \varepsilon(x_0^2+\delta x_1^2).
\end{equation}
Then $P^{\phi}=(1, 1+\delta^{\sigma/2}\varepsilon^q x^2, 1+\delta^{\sigma/2}\varepsilon^q y^2, \delta^{\sigma/2}\varepsilon^q x^2+\delta^{\sigma/2} \varepsilon^q y^2+\delta^{\sigma+1}z^2)$ with
\[(1+\delta^{\sigma/2}x_0^2)^{\sigma+2}+(1+\delta^{\sigma/2}x_0^2)(x_0^2+\delta x_1^2)\delta^{\sigma/2}+\delta(x_0^2+\delta x_1^2)^{\sigma}+\]
\[(1+\delta^{\sigma/2}y_0^2)^{\sigma+2}+(1+\delta^{\sigma/2}y_0^2)(y_0^2+\delta y_1^2)\delta^{\sigma/2}+\delta(y_0^2+\delta y_1^2)^{\sigma}=\]
\[1+\delta^{\sigma+1}x_0^{2(\sigma+2)}+\delta^{\sigma}x_0^4+\delta x_0^{2\sigma}+\delta^{\sigma/2}x_0^2+\delta^{\sigma}x_0^4+\delta^{\sigma/2+1}x_1^2+\delta^{\sigma+1}x_0^{2}x_1^2+\delta^{\sigma+1}x_1^{2\sigma}+\delta x_0^{2\sigma}+\]
\[1+\delta^{\sigma+1}y_0^{2(\sigma+2)}+\delta^{\sigma}y_0^4+\delta y_0^{2\sigma}+\delta^{\sigma/2}y_0^2+\delta^{\sigma}y_0^4+\delta^{\sigma/2+1}y_1^2+\delta^{\sigma+1}y_0^2y_1^2+\delta^{\sigma+1}y_1^{2\sigma}+\delta y_0^{2\sigma}=\]
\[\delta^{\sigma+1}x_0^{2(\sigma+2)}+\delta^{(\sigma+1)}x_0^2 x_1^2+\delta^{(\sigma+1)}x_1^{2\sigma}+\delta^{\sigma+1}y_0^{2(\sigma+2)}+\delta^{(\sigma+1)}y_0^2 y_1^2+\delta^{(\sigma+1)}y_1^{2\sigma}+\]
\[\delta^{\sigma/2}x_0^2+\delta^{\sigma/2+1}x_1^2+\delta^{\sigma/2}y_0^2+\delta^{\sigma/2+1}y_1^2=\]
\[=\delta^{\sigma+1}z_1^2+\delta^{\sigma/2}x_0^2+\delta^{\sigma/2+1}x_1^2+\delta^{\sigma/2}y_0^2+\delta^{\sigma/2+1}y_1^2. \]

Thus, from \eqref{rot2},\eqref{rot3} we obtain
\[\Gamma_{\varepsilon}(1+\delta^{\sigma/2}\varepsilon^q x^2)+\Gamma_{\varepsilon}(1+\delta^{\sigma/2}\varepsilon^q y^2)=Tr(\delta^{\sigma/2}x^2+\delta^{\sigma/2}y^2+\delta^{\sigma+1}z^2)\text{,}\]
that is $P^{\phi}\in \mathcal{H}^3_\varepsilon$. Similar computations show that $\phi$ leaves invariant the set of  points of 
the Hermitian cone $\mathcal{H}^3_{\varepsilon,\infty}$. Thus, $(\mathcal{H}^3_{\varepsilon})^{\phi}=\mathcal{H}^3_{\varepsilon}$.

Now, observe that $\phi=\alpha\beta$ where $%
\alpha :(x_{0},x_{1},x_{2},x_{3})\rightarrow
(x_{0}^{2},x_{1}^{2},x_{2}^{2},x_{3}^{2})$ and 
\[
\beta _{\delta }=\left( 
\begin{array}{cccc}
1 & 1 & 1 & 0 \\ 
0 & \delta ^{\sigma /2}\varepsilon ^{q} & 0 & \delta ^{\sigma /2}\varepsilon^q\\ 
0 & 0 & \delta ^{\sigma /2}\varepsilon ^{q} & \delta ^{\sigma /2}\varepsilon^q \\ 
0 & 0 & 0 & \delta ^{\sigma +1}%
\end{array}%
\right) 
\]
Thus $\left( \alpha \beta _{\delta }\right)
^{i}=\alpha ^{i}\prod_{j=0}^{2i}\beta _{\delta }^{\alpha ^{j}}$, and hence $2e$ is minimal integer $j$ such that $(\alpha \beta _{\delta })^{j}\in A$. Therefore, $%
2e$ divides $\left\vert Aut(\mathcal{H}_{\varepsilon }^3):A\right\vert $. On
the other hand, $\left\vert Aut(\mathcal{H}_{\varepsilon }^3):A\right\vert $
divides $\left\vert P\Gamma L_{4}(q^{2}):PGL_{4}(q^{2})\right\vert =2e$ and
so $\left\vert Aut(\mathcal{H}_{\varepsilon }^3):A\right\vert =2e$. That is to
say $Aut(\mathcal{H}_{\varepsilon }^3)=\left\langle E^{3},\vartheta ,\phi
\right\rangle $, which has order $4eq^{3}$. Finally, if $\tau_{\gamma_{1},\gamma_{2},\gamma_{3}} \in E^{3}$, one has
\begin{eqnarray*}
\tau_{\gamma_{1},\gamma_{2},\gamma_{3}}^{\phi_{\delta}}&=&\tau_{\delta^{\frac{1}{2}\sigma+1} \gamma _{1}^{2},\delta^{\frac{1}{2}\sigma +1}\gamma _{2}^{2},
\delta ^{\sigma +1}\gamma _{3}^{2}+\delta ^{\frac{1}{2}\sigma+1
}\gamma _{2}^{2}+ \delta ^{\frac{1}{2}\sigma +1}\gamma
_{1}^{2}+\delta^{\sigma+2 }\gamma _{1}^{2\sigma }+\delta^{\sigma+2}\gamma _{2}^{2\sigma}}\\
\vartheta^{\phi_{\delta}}&=&\vartheta \text{.}
\end{eqnarray*}
Thus $Aut(\mathcal{H}_{\varepsilon }^{3})=E^{3}\left\langle \vartheta, \phi \right\rangle$, which is of order $4eq^{3}$.

\end{proof}

\normalsize

\section{Minimal linear codes and some open problems}
 Let $\Omega$ be a set  of
points in $\PG(r, q)$. A linear code $C(\Omega)$ of length  $|\Omega|$ and  dimension $r+1$  can be constructed by taking as its  generator
matrix the one whose columns are the homogeneous coordinates of the points in $\Omega$.

An $[n, r+1, d]$--linear code $C$ is defined to be minimal if all of its codewords are minimal,  meaning that each codeword is uniquely determined, up to a nonzero scalar multiple, by its support.
Minimal codes are relevant in 
applications such as secret-sharing schemes and secure two-party computation; see
\cite{MM, CMP}.

Two codes are considered equivalent if one can be obtained from the other by a combination of coordinate permutations and scalar multiplications. 
This notion allows researchers to classify codes up to equivalence, significantly reducing the complexity of the classification problem. Demonstrating that two codes are equivalent implies that they possess the same structural and combinatorial properties, which is particularly useful when generalizing results or identifying canonical representatives within a family of codes.

In \cite{ALGS}, the authors studied the projective linear code $C=C(\mathcal{V}^r_{\varepsilon})$ induced by the projective system of the $\GF{2^{2e}}$-rational points of the variety 
$\mathcal{V}^r_{\varepsilon} \subset \PG(r,q^2)$, with $r \geq 3$ , defined by the  affine equation \eqref{eq:V}.
They determined their weights and proved that the code $\cC(\mathcal{V}^3_{\varepsilon})$ is a minimal linear code if and  only if $e\equiv 3 \pmod 4$, whereas  $\cC(\mathcal{V}^r_{\varepsilon})$  with $r>3$ is minimal for every odd $e>1$.

 Now, consider the projectivity $\xi $ of $PG(r,q^2)$ as defined in Theorem \ref{teofi}. It is easy to see that $\xi$ fixes 
\[\mathcal{V}_{\varepsilon_1,\infty}^r=\mathcal{V}_{\varepsilon_1}^r \cap \Pi_{\infty}=\mathcal{V}_{\varepsilon_2,\infty}^r.\]
Since $\mathcal{H}^r_{\varepsilon_1} \cap \AG(r,q^2)=\mathcal{V}^r_{\varepsilon_1}\cap \AG(r,q^2)$, the projectivity $\xi$
also maps $\mathcal{V}_{\varepsilon_{1}}^r$ onto $\mathcal{V}_{\varepsilon _{2 }}^r$%
. In particular,  this implies that the codes  $C(\mathcal{V}_{\varepsilon_{1}}^r)$ and $C(\mathcal{V}_{\varepsilon
_{2 }}^r)$ studied in \cite{ALGS} are  equivalent.

\bigskip

A natural direction for further research is the determination of the full automorphism group of the BT quasi-Hermitian varieties in the projective space $\PG(r, q^2)$ for all dimensions $r > 3$. A related question concerns the automorphism groups of the BM quasi-Hermitian varieties in $\PG(r, q^2)$ for $r > 2$ and arbitrary prime powers $q$ defined in \cite{ACK}. Notably, the case $r = 3$ with $q$ even has already been settled in \cite{AGMS}.

\end{document}